\documentclass[11pt]{amsart}
\usepackage{graphicx}
\usepackage{float}
\usepackage{amsfonts, amsmath, amsthm, amssymb}
\usepackage{enumerate}

\newcommand{\B}{\mathbb{B}}

\newcommand{\J}{\mathbb{J}}
\newcommand{\N}{\mathbb{N}}

\newcommand{\Q}{\mathbb{Q}}

\newtheorem{theorem}{{\bf Theorem}}

\newtheorem{corollary}{{\bf Corollary}}

\newtheorem{proposition}{\noindent {\bf Proposition}}

\newtheorem{lemma}{\noindent {\bf Lemma}}

\newtheorem{claim}{\noindent {\bf Claim}}
\newtheorem{question}{\noindent {\bf Question}}
\newtheorem{questions}[question]{Questions}

\def\endproof{\hfill {\kern 6pt\penalty 500
   \raise -0pt\hbox{\vrule \vbox to5pt {\hrule width 5pt
  \vfill\hrule}\vrule}}}
\makeindex
\begin{document}
\title[The length of chains in algebraic lattices]{The length of chains in  algebraic lattices}\label{chap:wellfoundedbis}
\author{Ilham Chakir }
\address{Math\'ematiques, Universit\'e Hassan $1^{er}$,
 Facult\'e des Sciences et Techniques,
 Settat, Maroc}
 \email {ilham.chakir@univ-lyon1.fr}
 \author{Maurice Pouzet}\address{Math\'ematiques, ICJ, Universit\'e Claude-Bernard Lyon1, 43 Bd 11 Novembre 1918, 69622 Villeurbanne Cedex, France}
\email{pouzet@univ-lyon1.fr}

\thanks{Research done under the auspices of  INTAS  project Universal Algebra and Lattice Theory}



\date{\today }

\begin{abstract}We study how the existence in an
algebraic lattice $L$ of a chain of a given type  is reflected in the join-semilattice $K(L)$ of
its compact elements. We show that for every  chain $\alpha$ of size
$\kappa$, there is a set $\B$ of  at most $2^{\kappa}$
join-semilattices, each one having  a least element such that an
algebraic lattice $L$ contains no chain of order type $I(\alpha)$ if
and only if the join-semilattice $K(L )$ of its compact elements
contains no join-subsemilattice isomorphic to a member of $\B$. We
show that among the join-subsemilattices of $[\omega]^{<\omega}$
belonging to $\B$, one is embeddable in all the others. We
conjecture that if $\alpha$ is countable, there is a finite   set $\B$.
\end{abstract}
\maketitle

 \section*{Introduction}
 This paper is about the relationship between the length of chains in an algebraic lattice $L$ and the structure of the join-semilattice $K(L)$ of the compact elements of $L$.  

We started such an investigation  in \cite{chak}, \cite{cp}, \cite{chapou}, \cite{chapou2}. Our motivation came from posets.
Let $P$ be an ordered set (poset). An \emph{ideal} of $P$ is any non-empty up-directed initial segment of $P$.  The set $J(P)$  of ideals of $P$, ordered by inclusion, is an interesting poset associated with $P$, and it is natural to ask about the relationship between the two posets.  For a concrete  example, if   $P:= [\kappa]^{<\omega}$   the set, ordered by inclusion, consisting of finite subsets of a set of size
$\kappa$, then $J([\kappa]^{<\omega})$ is isomorphic to $\mathfrak{P}(\kappa)$  the power set of $\kappa$ ordered by inclusion.  In \cite{chapou} we proved:
\begin{theorem} \label{thm0}
A poset $P$ contains a subset isomorphic to
$[\kappa]^{<\omega}$ if and only if $J(P)$ contains a subset isomorphic to
$\mathfrak {P}(\kappa)$.
\end{theorem}
Maximal chains in $\mathfrak {P}(\kappa)$ are of the form  $I(C)$, where  $I(C)$ is the chain of initial segments of an arbitrary chain $C$ of size $\kappa$
(cf. \cite {bonn-pouz1}). Hence, if $J(P)$ contains a subset isomorphic to
$\mathfrak{P}(\kappa)$ it contains a  copy of $I(C)$ for every  chain $C$ of size $\kappa$, whereas chains in $P$ can be
 small: eg in $P:=  [\kappa]^{<\omega}$ they are finite or have order type $\omega$. What
happens if for a given order type $\alpha$, particularly a countable one,  $J(P)$ contains no chain of type $\alpha$? A partial answer was given by Pouzet, Zaguia, 1984 (cf. \cite{pz} Theorem 4, pp.62).  In order to state their result,  we recall that the  order type $\alpha$ of a chain $C$ is \emph{indecomposable} if $C$ is embeddable in $I$ or in $C\setminus I$ for  every initial segment $I$ of  $C$.

\begin{theorem} \label{thm2}\footnote{In Theorem 4 \cite{pz}, $I(\alpha)$ is replaced  by $\alpha$.  This is due to the fact that  if $\alpha$ is a countable indecomposable order type and $P$ is a poset, $I(\alpha)$ can be embedded into $ J(P)$ if and only if $\alpha$ can be embedded  into $J(P)$. } Given an indecomposable countable order type $\alpha$, there is a finite list of ordered sets $A_{1}^{\alpha}, A_{2}^{\alpha},
\ldots, A_{n_{\alpha}}^{\alpha}$ such that for every poset $P$, the set $J(P)$ of ideals of $P$ contains no chain of type $I(\alpha)$
if and only if $P$ contains no subset isomorphic to one of the $A_{1}^{\alpha}, A_{2}^{\alpha}, \ldots,
A_{n_{\alpha}}^{\alpha}$.
\end{theorem}
Now, if $P$ is a join-semilattice with a least element, $J(P)$ is an algebraic lattice, and moreover every algebraic lattice is isomorphic to the poset $J(K(L))$ of ideals  of the join-semilattice $K(L)$ of the compact elements of $L$ (see \cite{grat}). Due to the importance of algebraic lattices, 
it was natural to ask whether  the two results above have an analog if posets are replaced by join-semilattices and subposets by  join-subsemilattices. This question was the starting  point of our research. 

We immediately observed that the specialization of Theorem \ref{thm0} to this case shows no difference.  Indeed {\it a  join-semilattice $P$ contains a subset isomorphic to
$[\kappa]^{<\omega}$ if and only if it contains a join-subsemilattice isomorphic to
$[\kappa]^{<\omega}$}. Turning to 
the specialization of Theorem \ref{thm2}, we notice that it as to be   quite different and is far from being  immediate.  In fact, we do not know yet whether   for every countable $\alpha$ there is a finite list as in Theorem \ref{thm2}. 

The purpose of this paper is to present the results obtained in that direction. In order to simplify the presentation, we will denote by $\mathbb{J}$ the class of join-semilattices having a least element. If $\mathbb{B}\subseteq \mathbb {J}$, we denote by $Forb( \mathbb B)$ the class of $P\in \mathbb{J}$ which contain no join-subsemilattice isomorphic to a member of  $\mathbb{B}$. If $\alpha$  denotes an order type, we denote by $\mathbb{J}_{\alpha}$ the class of members $P$ of $\mathbb {J}$ such that the lattice $J(P)$ of ideals of $P$ contains a chain of order type $I(\alpha)$. Finally we set $\mathbb{J}_{\neg \alpha}:= \mathbb{J}\setminus  \mathbb{J}_{\alpha}$.

Our first  result expresses that a characterization as  Theorem \ref{thm2} is possible.
\begin{theorem}\label{thmfirst}
For every order type  $\alpha$ there is a subset 
$\mathbb{B}$ of $\mathbb{J}$ of size at most  $2^{\mid\alpha\mid}$ 
such that $\mathbb{J}_{\neg \alpha}= Forb(\mathbb{B})$.
\end{theorem}

This is very weak. Indeed,  we cannot answer the following question.

\begin{question}
If $\alpha$ is countable, is there a finite  $\B$?
\end{question}

Examples are in order. Let $\omega$ be the order type of the chain $\N$ of non negative integers equipped with the natural order, $\omega^*$ be the order type of $\N$ equipped with the reverse order and   $\eta$ be the order type of the chain $\Q$ of rational numbers. Let  $\underline\Omega(\omega^*)$ be the join-semilattice obtained by adding a least element to the  set $[ \omega]^2$ of two-element subsets of $\omega$, identified to pairs  $(i,j)$, $i<j<\omega$, ordered so that
$(i,j)\leq (i',j')$ if and only if $i'\leq i$ and
$j\leq j'$ (see Figure 1).

\begin{figure}[htbp]
\centering
\includegraphics[width=3in]{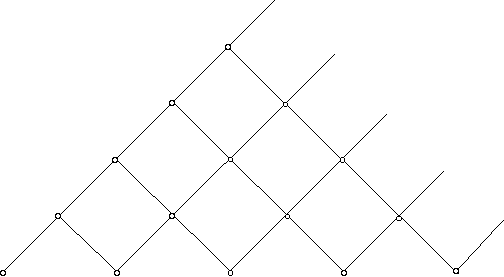}
\caption{$\underline {\Omega}(\omega^*)$}
\label{Omega}
\end{figure}

Furthermore, let $\underline{\Omega}(\eta)$ be the poset represented Figure \ref {Omegachap2}. It is immediate to see that if $\alpha$ is finite or $\omega$, then the list in Theorem  \ref{thm2} has just one member, namely $A_1^{\alpha}=\{\alpha'\}$ (where  $\alpha'$ is such  that $\alpha=1+\alpha'$). In this case, the specialization to join-semilattice yields the same result.  If $\alpha\in \{\omega^*,\eta\}$, it was proved in \cite {pz} that the list has two members: $\alpha$ and $\underline{\Omega}(\alpha)$. In this last case the specialization to join-semilattice is certainly different: the analogous list has at least three members, namely  $\alpha$,  $[\omega]^{<\omega}$ and $\underline\Omega(\alpha)$. If $\alpha=\omega^*$, these three members suffice. If $\alpha=\eta$, we do not know. These very specific cases take into account important classes of posets. Let us say  that a poset $P$ is \emph{well-founded}, resp. \emph{scattered},  if it contains no chain  of type $\omega^*$, resp. $\eta$. We have:
\begin{theorem}(Theorem 1.3 \cite{chapou2})\label{thm4chap2}
An algebraic lattice $L$  is well-founded if and only if
$K(L)$ is well-founded  and contains no join-subsemilattice  isomorphic to $\underline \Omega(\omega^*)$ or to $[\omega]^{<\omega}$.
\end{theorem}

\begin{figure}[htbp]
\centering
\includegraphics[width=2.5in]{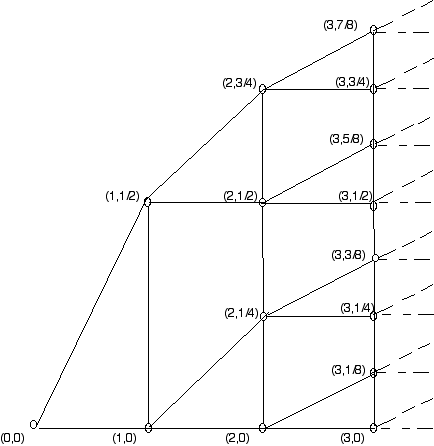}
\caption{$\underline \Omega(\eta)$}
\label{Omegachap2}
\end{figure}

\begin{question} Is it true that an algebraic lattice $L$  is scattered if and only if
$K(L)$ is scattered   and contains no join-subsemilattice  isomorphic to $\underline \Omega(\eta)$ or to $[\omega]^{<\omega}$?\end{question}

The join-semilattice  $[\omega]^{<\omega}$ never appeared in the list mentionned in  Theorem \ref {thm2} but,  as Theorem \ref{thm4chap2} illustrates, it might  appear in the specialization to join-semilattices. This raises two questions. For which $\alpha$ it appears? If it does not appears, are they particular join-semilattices of $[\omega]^{<\omega}$ which appear? Here are the answers:

\begin{theorem}\label{ordinal}
Let $\alpha$ be a countable order type. 
\begin{enumerate}[{(i)}]
\item  The join-semilattice $[\omega]^{<\omega}$ belongs  to every list $\mathbb B$ characterizing the class $\J_{\neg \alpha}$ of join-semilattices $P$ such that $J(P)$ contains no chain of type $I(\alpha)$  if and only if $\alpha$  is not an ordinal.
\item If $\alpha$ is an ordinal, then among the join-subsemilattices $P$ of  $[\omega]^{<\omega}$  which  does not belong to $\J_{\neg\alpha}$ there is one, say $Q_{\alpha}$,  which embeds as a join-semilattice in all the others.
\end{enumerate}
\end{theorem}
The poset $Q_{\alpha}$ is of the form $I_{<\omega}(S_{\alpha})$ where $S_{\alpha}$ is some sierpinskisation of $\alpha$ and $\omega$. We recall that a \emph{sierpinskisation} of a countable order type $\alpha$ and $\omega$, or simply of $\alpha$,  is any poset $(S, \leq)$ such that the order on $S$ is the intersection of two linear orders on $S$, one of type $\alpha$,  the other of type $\omega$. Such a  sierpinskisation can   be obtained from a bijective map $\varphi:\omega \rightarrow \alpha$, setting $S:=\N$ and $x\leq y$ if $x\leq y$ w.r.t. the natural order on  $\N$ and  $\varphi(x)\leq \varphi(y)$ w.r.t. the order of type  $\alpha$. Let $\omega.\alpha$ be the ordinal sum of $\alpha$ copies of the chain $\omega$; a sierpinskization  of $\omega.\alpha$ and $\omega$ is \emph{monotonic} if it obtained from a bijective map  $\varphi:\omega\rightarrow \omega\alpha$ such that $\varphi^{-1}$ is order-preserving on each subset of the form $\omega\times \{\beta\}$ where $\beta\in \alpha$. 
With that in hand, if $\alpha$ is an ordinal, we set $S_{\alpha}:=\alpha$ if $\alpha<\omega$. If   $\alpha=\omega\alpha'+n$ with  $\alpha'\not =0$ and  $n<\omega$,  let  $S_{\alpha}:=\Omega(\alpha')\oplus n$  be the direct sum of  $\Omega(\alpha')$ and the chain $n$, where  $\Omega(\alpha')$ is a monotonic sierpinskisation  of  $\omega\alpha'$ and $\omega$. We note that for countably infinite $\alpha$'s, $S_{\alpha}$ is a sierpinskisation of $\alpha$ and $\omega$.

Among the monotonic sierpinskisations of $\omega\alpha$ and $\omega$ there are some which are join-subsemilattices of the direct product $\omega\times \alpha$ that  we call  \emph{lattice sierpinskisations}. Indeed, to every  countable order type   $\alpha$, we may associate a join-subsemilattice $\Omega_L(\alpha)$ of  the direct product $\omega\times\alpha$ obtained via a monotonic sierpinskisation of $\omega\alpha$ and $\omega$. We add a least element, if there is none, and we denote by   $\underline\Omega_L(\alpha)$  the resulting poset. Posets  $\underline {\Omega}(\omega^*)$ and $\underline {\Omega}(\eta)$ represented Figure 1 and Figure 2  fall in this category.  We also associate the join-semilattice $P_{\alpha}$ defined as follows:

If  $1+\alpha\not \leq \alpha$, in which case   $\alpha= n+\alpha'$ with $n<\omega$ and  $\alpha'$ without a first element, we set $P_{\alpha}:=n+\underline \Omega_L(\alpha')$. If not, and  if $\alpha$ is equimorphic \index{equimorphic} to $\omega+ \alpha'$ we set $P_{\alpha}:= \underline \Omega_L(1+\alpha')$, otherwise, we set $P_{\alpha}= \underline \Omega_L(\alpha)$.

The importance of this kind of lattice sierpinskization steems from the following result:
 \begin{theorem}\label{thm:serpinskalpha}

 If $\alpha$ is countably infinite, $P_{\alpha}$ belongs to every list $\B$ characterizing 
 $\J_{\neg \alpha}$.
 \end{theorem}

This work leaves open the following questions. 
\begin{questions}
\begin{enumerate}
\item If  $\alpha$ is a countably infinite ordinal, does the minimal obstructions are $\alpha$, $P_{\alpha}$, $Q_\alpha$ and some  lexicographical sums of
obstructions corresponding to smaller  ordinal?
\item  If $\alpha$ is a scattered order type which is not an ordinal, does the minimal  obstructions are
$\alpha$, $P_{\alpha}$, $[\omega]^{<\omega}$ and some  lexicographical sums of
obstructions corresponding to smaller  scattered order types?
\end{enumerate}
\end{questions}

We only  have some examples of ordinals for which the answer to the first question is positive. We conjecture that the answer is always positive. 

To keep the paper at a reasonable length, we present only the proofs of Theorem \ref{thmfirst} and Theorem \ref{ordinal}. The proof of Theorem \ref{thm:serpinskalpha} and the presentation of examples supporting the conjecture are posponed to an other paper. The interested reader can find all the results mentionned here in \cite{chapou4} a paper available from the authors.
\section[Join-semilattices  and obstructions]{Join-semilattices and a proof of Theorem \ref{thmfirst}}
A \emph{join-semilattice} is a poset $P$ such that every two elements $x,y$ have a least upper-bound, or join, denoted by $x\vee y$. If $P$ has a least element, that we denote $0$, this amounts to say that every finite subset of $P$ has a join. In the sequel, we will mostly consider join-semilattices with a least element. Let $Q$ and $P$ be such join-semilattices. A map $f:Q\rightarrow P$ is \emph{join-preserving} if:
\begin{equation}\label{eq:def:joinpreserving}
f(x\vee y)=f(x)\vee f(y)
\end{equation}
for all $x,y\in Q$.

This map \emph{preserves finite (resp. arbitrary) joins} if
\begin{equation}\label{eq:def:finitepreserving}
f(\bigvee X)=\bigvee \{f(x): x\in X\}
\end{equation}
for every finite (resp. arbitrary) subset $X$ of $Q$.

If $P$ is a join-semilattice with a least element, the set $J(P)$ of ideals of $P$ ordered by inclusion is a complete lattice. If $A$ is a subset of  $P$, there is a least ideal  containing $A$, that we denote $<A>$. An ideal $I$ is \emph{generated} by a subset $A$ of $P$ if $I=<A>$.  If $[A]^{<\omega}$ denotes the collection of finite subsets of $A$ we have:
\begin{equation}\label{eq:def:generated}
<A>= \downarrow \{ \bigvee X: X\in [A]^{<\omega}\}
\end{equation}

\begin{lemma} \label{lem:arbitraryjoins}Let $Q$ be a join-semilattice with a least element and  $L$ be a complete lattice.  To a map  $g:Q\rightarrow L$ associate $\overline g: \mathfrak P(Q)\rightarrow L$ defined by setting  $\overline g(X):= \bigvee \{g(x): x\in X \}$ for every   $X \subseteq Q$. Then $\overline g$ induces a map from $J(Q)$ in $L$ which preserves arbitrary joins whenever $g$ preserves finite joins.  \end{lemma}
\begin{proof}

\begin{claim}\label{claim: supgenerates}Let $I\in J(Q)$. If $g$ preserves finite joins and $A$ generates $I$ then $\overline g(I) =\bigvee\{g(x): x\in A\}$.
\end{claim}
\noindent{\bf Proof of {claim} \ref{claim: supgenerates}.} Since $I=<A>$ and $g$ preserves finite joins, $<\{g(x): x\in I\}>=<\{g(x): x\in A\}>$. The claimed equality follows.
 \endproof

Now, let $\mathcal I\subseteq J(Q)$ and $I:=\bigvee \mathcal I$. Clearly, $A:= \bigcup \mathcal I$ generates $I$.   Claim \ref{claim: supgenerates}  yields  $\overline g(I) =\bigvee\{g(x): x\in A\}= \bigvee \bigcup \{\{g(x): x\in J\}: J\in \mathcal I\}=\bigvee \{\bigvee \{g(x): x\in J\}: J\in \mathcal I\}= \bigvee \{\overline g(J): J\in \mathcal I\}$. This proves that $\overline g$ preserves arbitrary joins. \end{proof}
%

\begin{lemma}\label{lem:11} Let $R$ be a poset and $P$ be a join-semilattice with a least element.
The following properties are equivalent:
\begin{enumerate} [{(i)}]
\item \label{item:0}There is an
embedding from $I(R)$ in $J(P)$ which preserves arbitrary joins.

\item \label{item:1}There is an
embedding from $I(R)$ in $J(P)$.

\item \label{item:2}There is a map $g$
from $I_{<\omega}(R)$ in $P$ such that
\begin{equation}  \label{eq:notsup}X\nsubseteq
Y_{1}\cup\ldots\cup Y_{n}   \Rightarrow g(X)\not\leq
g(Y_{1})\vee\ldots\vee g(Y_{n})
\end{equation} for all $X, Y_{1}, \ldots,
Y_{n}\in I_{<\omega}(R)$.

\item\label{item:3} There is a map $h:R \rightarrow
P$ such that \begin{equation} \label{eq:notleq}\forall i (1\leq i \leq
n\Rightarrow x\not\leq y_{i}) \Rightarrow h(x)\not\leq h(y_{1})\vee\ldots\vee h(y_{n})\end{equation} for all
$x, y_{1}, \ldots, y_{n}\in R$.
\end{enumerate}\end{lemma}

\begin{proof}\noindent $(\ref{item:0})\Rightarrow (\ref{item:1})$. Obvious.

 \noindent $(\ref{item:1})\Rightarrow (\ref{item:2})$  Let $f$ be an embedding from
$I(R)$ in $J(P)$. Let  $X \in I_{<\omega}(R)$. The set  $A:= Max(X)$ of maximal elements of $X$ is finite and $X:=\downarrow A$. Set $F(X):= \{f(R\setminus \uparrow a): a \in Max (X)\}$ and  $C(X):= f(X)\setminus \bigcup F(X)$.

\begin {claim} \label{claim:notempty}$C(X)\not =\emptyset$ for every $X \in I_{<\omega}(R)$\end{claim}
\noindent{\bf Proof of Claim \ref{claim:notempty}.}
 If $X= \emptyset$, $F(X)= \emptyset$. Thus  $C(X)=f(X)$ and our assertion is proved. We may then assume  $X\not = \emptyset$.  Suppose  $C(X)=\emptyset$, that is $f(X)\subseteq \bigcup F(X)$. Since $f(X)$ is an ideal of $P$ and  $\bigcup F(X)$ is a finite union of initial segments  of $P$, this implies that $f(X)$ is included in some, that is $f(X) \subseteq f(R\setminus \uparrow a)$ for some $a \in Max(X)$. Since $f$ is an embedding, this implies $X\subseteq  R\setminus \uparrow a$ hence $a\not \in X$. A contradiction.
\endproof

 Claim \ref{claim:notempty} allows us to pick an element $g(X)\in
C(X)$ for each $X \in I_{<\omega}(R)$. Let $g$ be the map defined by this process. We show that implication (\ref{eq:notsup}) holds. Let $X, Y_{1},\ldots,
Y_{n}\in I_{<\omega}(R)$. We have $g( Y_{i})\in f( Y_{i})$ for
every $1\leq i\leq n$. Since $f$ is an embedding
$f(Y_{i})\subseteq f( Y_{1}\cup \ldots \cup Y_{n})$ for every
$1\leq i\leq n$. But $f( Y_{1}\cup \ldots \cup Y_{n})$ is an
ideal. Hence $g(Y_{1})\vee\ldots\vee g(Y_{n})\in f( Y_{1} \cup
\ldots \cup Y_{n})$. Suppose $X\nsubseteq Y_{1}\cup \ldots \cup
Y_{n}$. There is $a\in Max(X)$ such that $Y_{1}\cup \ldots \cup
Y_{n}\subseteq R\setminus \uparrow a$.  And since $f$ is an embedding, $f(Y_{1}\cup \ldots \cup
Y_{n})\subseteq f(R\setminus \uparrow a)$. If $g(X)\leq g(Y_{1})\vee\ldots\vee g(Y_{n})$ then
$g(X)\in f( Y_{1}\cup \ldots\cup Y_{n})$. Hence $g(X)\in f(R\setminus \uparrow a)$,  contradicting
$g(X)\in C(X)$.

\noindent$(\ref{item:2})\Rightarrow (\ref{item:3})$.  Let $g: I_{<\omega}(R) \rightarrow P$ such that  implication (\ref{eq:notsup}) holds. Let  $h$ be the map induced by $g$ on  $R$ by setting $h(x):=g(\downarrow x)$ for $x\in R$. Let $x, y_{1}, \ldots,
y_{n}\in R$. If $x\not\leq y_{i}$ for every $1\leq i \leq n$, then
$\downarrow x \nsubseteq (\downarrow y_{1} \cup \ldots \cup
\downarrow y_{n})$. Since $g$ satisfies implication (\ref{eq:notsup}),
we have $h(x):=g(\downarrow x)\not\leq g(\downarrow y_{1})\vee
\ldots \vee g(\downarrow y_{n})=h(y_{1})\vee\ldots\vee
h(y_{n})$. Hence  implication (\ref{eq:notleq}) holds.

\noindent$(\ref{item:3})\Rightarrow (\ref{item:0})$ Let $h: R\rightarrow P$
such that implication (\ref{eq:notleq}) holds. Define $f: I(R)\rightarrow
J(P)$ by setting $f(I):=<\{h(x): x\in I\}>$, the ideal generated by
$\{h(x): x\in I\}$, for $I\in I(R)$. Since in $I(R)$ the join is the union,    $f$ preserves arbitrary joins.  We claim that $f$ is one-to-one.
Let $I, J\in I(R)$ such that $I\nsubseteq J$. Let $x\in
I\setminus  J$. Clearly $h(x)\in f(I)$. We claim that
$h(x)\not\in f(J)$. Indeed, if $h(x)\in f(J)$, then
$h(x)\leq\bigvee\{h(y): y\in F\}$ for some finite subset $F$ of
$J$. Since implication (\ref{eq:notleq}) holds, we have $x\leq
y$ for some $y\in F$. Hence $x\in J$, contradiction. Consequently
$f(I)\nsubseteq f(J)$. Thus  $f$ is one-to-one as claimed.
\end{proof}

The following proposition rassembles the main properties of the comparizon of join-semilattices.
\begin{proposition} \label{1.3} Let  $P$, $Q$ be two join-semilattices with a least element. Then:
\begin{enumerate}
\item \label{item 1lem1.3}$Q$ is embeddable in $P$ by a join-preserving map  iff $Q$ is embeddable in  $P$ by a map preserving finite joins.
\item \label{item 2lem1.3}If $Q$ is embeddable  in  $P$ by a join-preserving map then $J (Q)$ is embeddable  in  $J (P )$
by a map preserving arbitrary joins.

Suppose $Q := I_{<\omega} (R)$ for some poset $R$. Then:
\item  \label{item 3lem1.3}$Q$ is embeddable in $P$ as a poset iff $Q$ is embeddable in   $P$ by a map preserving finite joins.
\item \label{item 4lem1.3}$J (Q)$ is embeddable in $J (P )$ as a poset iff $J (Q)$  is embeddable in   $J(P)$ by a map preserving arbitrary  joins.
\item  \label{item 5lem1.3}If $\downarrow  x$ is finite for every $x \in  R$ then $Q$ is embeddable  in $P$ as a poset iff $J (Q)$
is embeddable in  $J (P )$ as a poset.
\end{enumerate}
\end{proposition}

\begin{proof}
\begin{enumerate}[{}]
\item (\ref{item 1lem1.3}) Let $f : Q \rightarrow  P$ satisfying $f (x \vee y) = f (x)
\vee f (y)$ for all $x, y \in  Q$. Set $g(x) := f (x)$  if
$x\not=0$ and $g(0) := 0$. Then g preserves finite joins.

\item (\ref{item 2lem1.3}) Let $f : Q \rightarrow  P$ and $\overline f : J (Q)
\rightarrow J (P )$ defined by $f (I ) :=\downarrow \{f (x) : x
\in I \}$. If $f$ preserves finite joins, then $\overline f$
preserves arbitrary joins. Furthermore, $\overline f$ is one to one provided that $f$ is one-to-one.

\item (\ref{item 3lem1.3}) Let $ f : Q \rightarrow P$ . Taking account that $Q :=
I_{<\omega} (R)$, set $g(\emptyset) := 0$ and $g(I ):= \bigvee \{f
(\downarrow x) : x \in I \}$ for each $I \in   I_{<\omega}
(R)\setminus \{\emptyset\}$. Since in $Q$ the join is the union, the map $g$ preserves finite unions.

\item (\ref{item 4lem1.3}) This is equivalence $(\ref{item:0})\Longleftrightarrow (\ref{item:3})$ of Lemma \ref{lem:11}.
\item (\ref{item 5lem1.3}) If $Q$ is embeddable in  $P$ as a poset, then from  Item (\ref{item 3lem1.3}), $Q$ is embeddable in
 $P$ by a map preserving finite joins. Hence from Item $(\ref{item 1lem1.3})$, $J(Q)$
is embeddable  in $J(P)$ by a map preserving arbitrary joins. Conversely, suppose that $J(Q)$ is embeddable in $J(P)$ as
a poset. Since $J(Q)$ is isomorphic to  $I(R)$, Lemma \ref{lem:11} implies that
there is map $h$ from $R$ in $P$ such that implication
$(\ref{eq:notleq})$ holds.  According to the proof of Lemma \ref{lem:11}, the map $f: I(R) \rightarrow J(P)$ defined by setting
$f(I):=\bigvee\{h(x): x\in I\}$ is an embedding preserving arbitrary joins. Since $\downarrow x$ is finite for every $x\in R$,  $I$ is  finite , hence $f(I)$ has a largest element, for   every $I\in I_{<\omega}(R)$. Thus $f$ induces an embedding from $Q$ in $P$ preserving finite joins. \end{enumerate}
\end{proof}

\begin{theorem} \label{1.4} Let $R$ be a poset, $Q := I_{<\omega} (R)$ and $\kappa:=\vert Q\vert$. Then, there is a set $\B$, of size
at most $2^\kappa$, made of join-semilattices, such that for every join-semilattice
$P$, the join-semilattice
$J (Q)$ is not embeddable  in $J (P )$ by a map preserving arbitrary joins if and only
if no member $Q_f$ of $\B$ is embeddable in $P$ as a join-semilattice.
\end{theorem}

\begin{proof}
If   $\downarrow x$ is finite for every $x\in R$ the conclusion of the theorem holds with $\mathbb B= \{ Q\}$ (apply Item (\ref{item 3lem1.3}), (\ref{item 4lem1.3}), (\ref {item 5lem1.3}) of Proposition \ref{1.3}). So we may assume that $R$ is infinite. Let $P$ be a join-semilattice. Suppose that there is an embedding $f$ from $J (Q)$ in $J (P )$ which  preserves arbitrary
joins.
\begin{claim}\label{claim:counting}There is a join-semilattice $Q_f$ such that
\begin{enumerate}
\item $Q_f$
embeds in $P$ as a join-semilattice.
\item $J (Q)$
embeds in $J (Q_f )$ by a map preserving arbitrary joins.
\item  $\vert Q_f \vert =
\vert Q\vert$.
\end{enumerate}
\end{claim}
\noindent{\bf Proof of Claim \ref{claim:counting}.}
 From Lemma \ref{lem:11}, there is a map $g : Q
\rightarrow P$  such that inequality   $(\ref{eq:notsup})$ holds. Let $Q_f $
be the join-semilattice of $P$ generated by $\{g(x) : x \in Q\}$. This inequality holds when $P$ is replaced by $Q_f$. Thus from Proposition  \ref{1.3}, $J (Q)$ embeds into $J (Q_f )$ by a map
preserving arbitrary joins. Since $R$ is infinite,  $\vert Q_f \vert =
\vert Q\vert=\vert R\vert$.
\endproof

For each join-semilattice $P$ and  each embedding $f:J (Q)\rightarrow J (P )$ select $Q_f$, given by Claim \ref{claim:counting}, on a fixed set of size $\kappa$.  Let $\B$ be the collection of this join-semilattices. Since the number of join-semilattices on a set of size $\kappa$ is at most $2^{\kappa}$, $\vert\mathbb B\vert \leq 2^\kappa$.
\end{proof}

\noindent{\bf Proof of Theorem \ref{thmfirst}.}
Let $\alpha$ be the order type of a chain $C$. Apply Theorem \ref{1.4}
 above with $R:=C$. Since,  in
 this case,  $J(Q)$ is embeddable  in $J(P)$ if and only if $J(Q)$ is embeddable in $J(P)$ by a map preserving arbitrary joins, Theorem  \ref{thmfirst}  follows. \endproof

\section [Sierpinskisations]{Sierpinskisations and a proof of Theorem \ref{ordinal} }
\subsection{Proof of item (i) of Theorem \ref{ordinal}}
We start with the following lemma
\begin{lemma}\label{lem:finitegenesierp} If $\alpha$ is a countably infinite order type and $S$ is a sierpinskisation of $\alpha$ and $\omega$ then the join-semilattice $I_{<\omega}(S)$, made of finitely generated initial segments of $S$, is isomorphic to a  join-subsemilattice of $[\omega]^{<\omega}$ and  belongs to $\J_{\alpha}$.\end{lemma}

\begin{proof}By definition,  the order on a sierpinskisation $S$ of $\alpha$  and $\omega$ has a
linear extension such that the resulting chain $\overline S$ has
order type $\alpha$. Hence, from a result of \cite{bonn-pouz1}, the chain $I(\overline S)$ is a maximal chain
of $I(S)$ of type $I(\alpha)$. The lattices $I(S)$ and
$J(I_{<\omega}(S))$ are isomorphic, thus $I_{<\omega}(S)\in
\J_{\alpha}$. The order on $S$ has  a linear extension of  type
$\omega$, thus  every principal initial segment of $S$ is finite and
more generally every finitely generated initial segment of $S$ is
finite. This tells us that $I_{<\omega}(S)$ is a join-subsemilattice
of $[S]^{<\omega}$. Since $S$  is countable, $I_{<\omega}(S)$
identifies to a join-subsemilattice of $[\omega]^{<\omega}$.
\end{proof}

The proof of the "only if" part goes as follows.  Let $S$ be a sierpinskization of $\alpha$ and $\omega$.  According to Lemma \ref{lem:finitegenesierp}, the join-semilattice $I_{<\omega}(S)$, made of finitely generated initial segments of $S$, is isomorphic to a  join-subsemilattice of $[\omega]^{<\omega}$ and  belongs to $\J_{\alpha}$. If  $[\omega]^{<\omega}$ is a minimal member of $\mathbb{J}_{\alpha}$, then $[\omega]^{<\omega}$ is   embeddedable as a join-subsemilattice in  $I_{<\omega}(S)$. To conclude that $\alpha$ cannot be an ordinal, it suffices to prove: 

\begin{lemma} \label{lem:wqo}If $\alpha$ is  an ordinal and $S$ is a sierpinskisation of $\alpha$ and $\omega$, then  $[\omega]^{<\omega}$ is not embeddable in  $I_{<\omega}(S)$.
\end{lemma}

This simple fact relies on the important notion of
well-quasi-ordering introduced by Higman \cite{higm}. We recall that
a poset $P$ is \emph{well-quasi-ordered} (briefly w.q.o.) if every
non-empty subset $A$ of $P$ has at least a minimal element and the
number of these minimal elements is finite. As shown by Higman, this
is equivalent to the fact that $I(P)$ is well-founded \cite{higm}.

 Well-ordered set are trivially w.q.o. and, as it is well known, the direct product of finitely many w.q.o. is w.q.o.  Lemma \ref{lem:wqo} follows immediately from this. Indeed, if   $S$ is a sierpinskisation of $\alpha$ and $\omega$, it embeds in the direct product $\omega\times \alpha$.  Thus $S$  is w.q.o. and consequently  $I(S)$  is well-founded. This implies that $[\omega]^{<\omega}$ is not  embeddable in  $I_{<\omega}(S)$. Otherwise  $J([\omega]^{<\omega})$ would be embeddable in $J(I_{<\omega}(S))$, that is  $\mathfrak P(\omega)$ would be embeddable in $I(S)$. Since $\mathfrak P(\omega)$ is not well-founded, this would contradict the well-foundedness of $I(S)$.

The "if" part is based on our earlier work on well-founded algebraic lattices, and essentially on the following corollary of Theorem \ref{thm4chap2}.

 \begin{theorem} \label{thm:posetchap2}(Corollary 1.1, \cite{chapou2})  A join-subsemilattice $P$ of $[\omega]^{<\omega}$ contains either  $[\omega]^{<\omega}$ as a join-semilattice or is well-quasi-ordered. In the latter case,   $J(P)$ is well-founded.
\end{theorem}

With this result, the proof of the "if" part of Theorem \ref{ordinal} is immediate.  Indeed, suppose  that $\alpha$ is not an ordinal.
Let $P\in\mathbb{J}_{\alpha}$. The lattice $J(P)$ contains a chain
isomorphic to $I(\alpha)$. Since $\alpha$ is not an ordinal,
$\omega^{*}\leq\alpha$. Hence, $J(P)$ is not well-founded. If $P$ is embeddable  in $[\omega]^{<\omega}$ as a
join-semilattice then, from
Theorem \ref{thm:posetchap2},  $P$ contains a join-subsemilattice
isomorphic to $[\omega]^{<\omega}$. Thus  $[\omega]^{<\omega}$ is
minimal in $\mathbb{J}_{\alpha}$.

A sierpinskisation $S$ of a countable order type $\alpha$ and $\omega$ is embeddable into $[\omega]^{<\omega}$ as a poset.  A consequence of Theorem \ref{thm:posetchap2} is the following
\begin{corollary}\label{lem:sierpnotinw} If  $S$  can be  embedded in $[\omega]^{<\omega}$ as a join-semilattice, $\alpha$ must be an ordinal.
\end{corollary}
\begin{proof} Otherwise,  $S$ contains an infinite antichain and by Theorem \ref{thm:posetchap2} it contains a copy of $[\omega]^{<\omega}$. But this poset cannot be embedded in a sierpinskisation. Indeed,  a sierpinskisation is embeddable  into a product of two chains, whereas $[\omega]^{<\omega}$ cannot be embedded in a product of finitely many chains (for every integer $n$, it contains the power set $\mathfrak P(\{0,\dots ,n-1\})$ which cannot be embedded into  a product of less than $n$ chains;  its dimension, in the sense of Dushnik-Miller's notion of dimension, is infinite, see \cite{trotter}).\end{proof}

\subsection{Proof of item (ii) of Theorem \ref{ordinal}}

We prove first that there is a sierpinskisation $S$ of $\alpha$ and $\omega$ such that $Q:= I_{<\omega}(S)\in \J_{\alpha}$  is embeddable in $P$ by a map preserving finite joins.

\begin{theorem} \label{thm:qalphaprelim}Let $\alpha$ be a countable ordinal and $P\in \J_{\alpha}$. If $P$ is embeddable in $[\omega]^{<\omega}$ by a map preserving finite joins there is a sierpinskisation $S$ of $\alpha$ and $\omega$ such that $I_{<\omega}(S)\in \J_{\alpha}$ and $I_{<\omega}(S)$ is embeddable in $P$ by a map preserving finite joins.
\end{theorem}
\begin{proof}
We construct first $R$  such that $I_{<\omega}(R)\in \J_{\alpha}$ and $I_{<\omega}(R)$ is embeddable in $P$ by a map preserving finite joins.

We may suppose that $P$ is a subset of $[\omega]^{<\omega}$ closed under finite unions. Thus   $J(P)$ identifies with  the set of arbitrary unions of members of $P$.
Let $(I_{\beta})_{\beta<\alpha+1}$ be  a strictly increasing sequence of ideals of $P$. For each $\beta<\alpha$ pick $x_{\beta}\in I_{\beta+1}\setminus I_{\beta}$ and $F_{\beta}\in P$ such that $x_\beta \in F_{\beta}\subseteq  I_{\beta+1}$.  Set $X:= \{x_{\beta}: \beta<\alpha\}$, $\rho:=\{(x_{\beta'},x_{\beta''}): \beta'<\beta''<\alpha\;  \text{and}\;  x_{\beta'}\in F_{\beta''}\}$. Let  $\hat \rho$ be the reflexive transitive closure of $\rho$. Since $\theta:= \{(x_{\beta'},x_{\beta''}): \beta'<\beta''<\alpha\}$ is a linear order containing $\rho$, $\hat \rho$ is an order on $X$. Let   $R:=(X, \hat \rho)$ be the resulting poset.
\begin{claim}\label{claim:jalpha} $I_{<\omega}(R)\in \J_{\alpha}$.
\end{claim}
\noindent { \bf Proof of claim \ref{claim:jalpha}.} The linear order
$\theta$ extends the order $\hat\rho$ and has type $\alpha$, thus
$I(R)$ has a maximal chain of type $I(\alpha)$. Since
$J(I_{<\omega}(R))$ is isomorphic to $I(R)$, $I_{<\omega}(R)$
belongs to $\J_{\alpha}$ as claimed.\endproof

\begin{claim}\label{claim:finite} For each $x\in X$, the initial segment $\downarrow x$ in  $R$ is finite.
\end{claim}
\noindent { \bf Proof of claim \ref{claim:finite}.} Suppose not. Let $\beta$ be minimum such that
for $x:=x_{\beta}$, $\downarrow x$ is infinite.
For each $y\in X$ with $y<x$ in $R$ select a finite sequence $(z_i(y))_{i\leq n_y}$ such that:
\begin{enumerate}
\item $z_{0}(y)=x$ and $z_{n_y}=y$.
\item \label{item:finite2}$(z_{i+1(y)}, z_{i}(y))\in \rho$ for all $i<n_{y}$.
\end{enumerate}
According to item \ref{item:finite2}, $z_{1}(y)\in F_{\beta}$. Since $F_\beta$ is finite, it contains some $x':=x_{\beta'}$ such that $z_{1}(y)= x'$ for infinitely many $y$. These elements belong to $\downarrow x'$. The fact that  $\beta'<\beta$  contradicts the choice of $x$.
\endproof

\begin{claim} \label{claim:union} Let $\phi$ be defined by  $\phi(I):= \bigcup \{F_{\beta}: x_\beta\in I\}$
for each $I\subseteq X$. Then:
\item $\phi$ induces an embedding  of $I(R)$ in $ J(P)$ and an embedding of $I_{<\omega}(R)$ in $P$.
\end{claim}
\noindent { \bf Proof of claim \ref{claim:union}.} We prove the
first part of the claim. Clearly, $\phi (I)\in J(P)$ for each
$I\subseteq X$. And trivially,  $\phi$ preserves arbitrary unions.
In particular, $\phi$ is order preserving. Its remains to show that
$\phi$ is one-to-one. For that, let $I,J\in I(R)$ such that
$\phi(I)=\phi(J)$. Suppose $J\not \subseteq I$.  Let $x_{\beta}\in
J\setminus I$, Since $x_\beta \in J$, $x_\beta\in F_{\beta}\subseteq
\phi(J)$. Since  $\phi(J)= \phi(I)$, $x_{\beta}\in \phi(I)$. Hence
$x_\beta \in F_{\beta'}$ for some $\beta'\in I$. If $\beta'<\beta$
then since $F_{\beta'}\subseteq I_{\beta'+1}\subseteq I_{\beta}$ and
$x_{\beta}\not \in I_{\beta}$, $x_{\beta}\not \in F_{\beta'}$. A
contradiction. On the other hand, if  $\beta<\beta'$ then,  since
$x_\beta\in F_{\beta'}$, $(x_\beta, x_{\beta' })\in \rho$. Since $I$
is an initial segment of $R$, $x_\beta\in I$. A contradiction too.
Consequently $J\subseteq I$. Exchanging the roles of $I$ and $J$,
yields $I\subseteq J$.  The equality $I=J$ follows. For the second
part of the claim, it suffices to show that $\phi(I)\in P$ for every
$I\in I_{<\omega }(R)$. This fact is a straightforward consequence
of Claim \ref{claim:finite}. Indeed, from this claim $I$ is finite.
Hence $\phi(I)$ is finite and thus belongs to $P$.

\begin{claim} \label{claim:linear} The order $\hat \rho$ has a linear extension of type  $\omega$.
\end{claim}
\noindent { \bf Proof of claim \ref{claim:linear}.}  Clearly, $[\omega]^{<\omega}$ has a linear extension of type $\omega$.  Since $R$ embeds  in $[\omega]^{<\omega}$, via an embedding  in $P$, the induced linear extension on $R$ has order type $\omega$.\endproof

Let $\rho'$ be the intersection of such a linear extension with the order $\theta$ and let $S:= (X, \rho')$.

\begin{claim} \label{claim:newembedding} For every $I\in I(S)$, resp. $I\in I_{<\omega}(S)$ we have $I\in I(R)$, resp. $I\in I_{<\omega}(R)$. \end{claim}
\noindent { \bf Proof of claim \ref{claim:newembedding}.} The first part of the proof follows directly from the fact that $\rho'$ is a linear extension of $\hat \rho$. The second part follows from the fact that each $I\in I_{<\omega}(S)$ is finite.
\endproof

It is then easy to check that the poset $S$ satisfies the properties stated in the theorem.\end{proof}

In order to conclude, it suffices to prove that one can replace $Q$ by  $Q_{\alpha}:=I_{<\omega}(S_{\alpha})$,  where  $S_{\alpha}$ is the sierpinskization  defined in the introduction. 

This fact follows directly from Lemma \ref{lem:qalphaprelim} below. It relies on properties of monotonic sierpinskizations, some already in  \cite{pz}. 

 We recall that  for a countable order type $\alpha'$, two monotonic sierpinskisations of $\omega\alpha'$ and $\omega$ are embeddable in each other and denoted by the same symbol $\Omega(\alpha')$ and we recall the  following result (cf. \cite{pz} Proposition 3.4.6. pp. 168).

\begin{lemma}\label{lem: proppouzag}
Let $\alpha'$ be a countable order type. Then $\Omega(\alpha')$ is embeddable in every sierpinskisation $S'$ of $\omega\alpha'$ and $\omega$.
\end{lemma}

\begin{lemma}\label{lem:trivial}
Let $\alpha$ be a countably infinite order type and $S$ be a sierpinskisation of $\alpha$ and $\omega$.
Assume that $\alpha= \omega\alpha'+n$  where $n<\omega$. Then there is a subset of $S$ which is the direct sum $S'\oplus F$ of a sierpinskisation $S'$ of  $\omega\alpha'$ and $\omega$ with an $n$-element poset $F$.
\end{lemma}
\begin{proof} Assume that $S$ is given by a bijective map  $\varphi$  from $\N$ onto a chain $C$ having order type $\alpha$. Let $A'$ be the set of the $n$ last elements of $C$,  $A:=\varphi^{-1}(A')$  and   $a$ be the largest element of $A$ in $\N$. The image of $]a \rightarrow )$ has order type $\omega\alpha'$, thus $S$ induces on $]a \rightarrow )$ a sierpinskisation $S'$ of $\omega\alpha'$ and $\omega$.  Let $F$ be the poset induced by $S$ on $A$. Since every  element of $S'$ is incomparable to every element of $F$
 these  two posets form a direct sum. \end{proof}

 Let  $\alpha$ be   a countably infinite order type such that 
$\alpha= \omega\alpha'+n$  where $n<\omega$.  We set $S_{\alpha}:=\Omega(\alpha')\oplus n$ and $Q_{\alpha}:=
I_{<\omega}( S_{\alpha})$.
 
\begin{lemma}\label{lem:qalphaprelim} $Q_{\alpha}\in \mathbb J_{\alpha}$ and for every sierpinskisation $S$ of $\alpha$ and $\omega$, $Q_{\alpha}$
 is embeddable in $I_{<\omega}(S)$ by a map preserving finite joins.
\end{lemma}
\begin{proof} For the the first part, apply Lemma \ref{lem:finitegenesierp}.

{\bf Case 1.} $n=0$. By Lemma \ref{lem: proppouzag},   $
\Omega(\alpha')$ is embeddable in $S$. Thus $Q_{\alpha}$ is
embeddable in $I_{<\omega}(S)$ by a map preserving finite joins.

{\bf Case 2.} $n\not = 0$. Apply Lemma \ref{lem:trivial}. According to Case 1, $I_{<\omega}( \Omega(\alpha'))$ is embeddable in $I_{<\omega}(S')$. On an other hand $n+1$ is embeddable in $I_{<\omega}(F)=I(F)$.
Thus  $Q_{\alpha}$ which is isomorphic to the product $I_{<\omega}( \Omega(\alpha'))\times (n+1)$ is embeddable in  the product  $I_{<\omega}(S')\times I_{<\omega}(F)$. This product is   itself isomorphic to $I_{<\omega}(S'\oplus F)$. Since $S'\oplus F$ is embeddable in $S$, $I_{<\omega}(S'\oplus F)$ is embeddable in $I_{<\omega}(S)$ by a map preserving finite joins.  It follows that $Q_{\alpha}$ is embeddable in $I_{<\omega}(S)$ by a map preserving finite joins.
\end{proof}

\printindex


\begin{thebibliography}{999}

\bibitem{bonn-pouz1}R Bonnet, M Pouzet, Extensions et stratifications d'ensembles dispers\'es,
Comptes Rendus Acad. Sc.Paris, 268, S\'erie A, (1969),1512-1515.

\bibitem{chak} I.Chakir, Cha\^{\i}nes d'id\'eaux  et dimension alg\'ebrique des treillis distributifs, Th\`ese de doctorat, Universit\'e
Claude-Bernard(Lyon1) 18 d\'ecembre 1992, n° 1052.

\bibitem{cp} I.Chakir, M.Pouzet, The length of chains in distributive lattices, Notices of the A.M.S., 92 T-06-118, 502-503.

\bibitem{chapou} I.Chakir, M.Pouzet, Infinite independent sets in  distributive lattices, Algebra   Universalis {\bf 53}(2) 2005, 211-225.

\bibitem  {chapou2} I. Chakir, M. Pouzet, A characterization of well-founded algebraic lattices, submitted to Order, under revision.

\bibitem {chapou3} I.Chakir, M.Pouzet, The length of chains in algebraic modular lattices, ORDER {\bf 24}(4) (2007)227-247.
\bibitem {chapou4} I.Chakir, Conditions de cha\^{\i}nes dans les treillis alg\'ebriques, Universit\'e de Settat, Octobre 2007, 106 pp. fichier .pdf.
%

\bibitem{fraissetr}
R.~Fra{\"\i}ss{\'e}.
\newblock {\em Theory of relations}.
\newblock North-Holland Publishing Co., Amsterdam, 2000.


\bibitem  {grat}G.Gr\"atzer, General Lattice Theory, Birkh\"auser, Stuttgard, 1998.


\bibitem{higm}G. Higman, Ordering by divisibility in abstract algebras, Proc. London. Math. Soc. 2 (3),
(1952), 326-336.


\bibitem{pz} M. Pouzet and N. Zaguia, Ordered sets with no chains
of ideals of a given type, Order, {\bf 1} (1984), 159-172.


\bibitem  {trotter} W.T.~Trotter.
\newblock{\em  Combinatorics and Partially Ordered Sets: Dimension Theory,} The Johns
Hopkins University Press, Baltimore, MD, 1992.

%
\end{thebibliography}
\end{document}